\documentclass[a4paper]{article}
\usepackage[english]{babel}

\usepackage{amssymb, amsmath,amsthm}
\usepackage{mathrsfs}
\usepackage{amscd}
\usepackage[active]{srcltx}
\usepackage{verbatim}
\usepackage{graphicx}

\usepackage[backend=bibtex,style=alphabetic]{biblatex}
\usepackage[utf8]{inputenc}
\usepackage[T1]{fontenc}
\usepackage{lmodern}

\usepackage{hyperref} 
\usepackage{todonotes}
\usepackage{enumerate}
\usepackage{mathtools}
\usepackage{bbm}

\mathtoolsset{showonlyrefs,showmanualtags}

\newcommand\dd{\mathrm{d}}
\newcommand\ud{\,\mathrm{d}}

\newcommand{\Dd}{\mathcal{D}}

\newcommand{\Gg}{\mathcal{G}}
\newcommand{\Hh}{\mathcal{H}}

\newcommand{\Oo}{\mathcal{O}}

\newcommand{\Xx}{\mathcal{X}}

\newcommand{\MM}{\mathbb{M}}
\newcommand{\NN}{\mathbb{N}}

\newcommand{\PP}{\mathbb{P}}

\newcommand{\RR}{\mathbb{R}}

\newcommand\Tr{\mathrm{Tr}}

\newcommand{\inp}[2]{\langle #1,#2 \rangle}

\newcommand{\Ric}{\mathrm{Ric}}
\newcommand{\Nabla}{\nabla}

\newcommand{\ev}{\mathrm{ev}}

\DeclareMathOperator*{\LIM}{LIM}

\renewcommand{\epsilon}{\varepsilon}

\theoremstyle{plain}
\newtheorem{theorem}{Theorem}[section]
\theoremstyle{remark}
\newtheorem{remark}[theorem]{Remark}

\theoremstyle{plain}
\newtheorem{corollary}[theorem]{Corollary}
\newtheorem{lemma}[theorem]{Lemma}
\newtheorem{proposition}[theorem]{Proposition}
\newtheorem{definition}[theorem]{Definition}

\newtheorem{assumption}[theorem]{Assumption}

\numberwithin{equation}{section}

\allowdisplaybreaks

\begin{document}

\title{Large deviations for Brownian motion in evolving Riemannian manifolds}

\author{
\renewcommand{\thefootnote}{\arabic{footnote}}
Rik Versendaal\footnotemark[1]
}

\footnotetext[1]{
Delft Institute of Applied Mathematics, Delft University of 
Technology, P.O. Box 5031, 2600 GA Delft, The Netherlands, E-mail: \texttt{R.Versendaal@tudelft.nl}.
}

\date\today


\maketitle

\begin{abstract}
We prove large deviations for $g(t)$-Brownian motion in a complete, evolving Riemannian manifold $M$ with respect to a collection $\{g(t)\}_{t\in\RR}$ of Riemannian metrics, smoothly depending on $t$. We show how the large deviations are obtained from the large deviations of the (time-dependent) horizontal lift of $g(t)$-Brownian motion to the frame bundle $FM$ over $M$. The latter is proved by embedding the frame bundle into some Euclidean space and applying Freidlin-Wentzell theory for diffusions with time-dependent coefficients, where the coefficients are jointly Lipschitz in space and time.\\

\textit{keywords:} large deviations, Schilder's theorem, $g(t)$-Brownian motion, time-dependent geometry, evolving manifold, frame bundle, horizontal lift, anti-development, Freidlin-Wentzell theory
\end{abstract}

\tableofcontents

\section{Introduction}

In the past decades, the study of evolving Riemannian manifolds has received a lot of attention. The treatment of stochastic processes in this setting was initiated in \cite{ACT08}, where Brownian motion with respect to a collection of time-dependent metrics is defined. The existence of this process is proven, and the gradient of the associated heat-semigroup is studied when the metric evolves under the Ricci-flow. This is further developed in \cite{Cou11}. More generally, in \cite{GPT15}, the theory of martingales with respect to a time-dependent connection is studied.

In \cite{Cou14}, the so-called Onsager-Machlup functional is studied for elliptic diffusions on manifolds with time-dependent metric. It is shown that the probability that a Brownian motion deviates from a smooth curve by at most a distance $\epsilon > 0$ decays exponentially in $\epsilon$. More precisely, if $X_t$ is a Brownian motion with respect to a time-dependent metric $\{g(t)\}_{0\leq t \leq 1}$, and $\gamma:[0,1] \to M$ is a smooth curve, it is proven that for $\epsilon$ small
\begin{multline}
\PP\left(\sup_{0\leq t\leq 1} d_t(X_t,\gamma(t)) \leq \epsilon\right)\\
 \sim e^{-\frac{C}{\epsilon^2}}\exp\left\{\int_0^1 -\frac12|\dot\gamma(t)|_{g(t)}^2 - \frac{1}{12}R_{g(t)}(\gamma(t)) + \frac14\Tr_{g(t)}(\partial_1g(t))\ud t\right\}.
\end{multline}
Here, $d_t$ is the Riemmanian distance associated to the metric $g(t)$, $R_g(t)$ is the scalar curvature of the metric $g(t)$, and $\Tr_{g(t)}(\partial_1g(t))$ denotes the trace of the time-derivative $\partial_1g(t)$ with respect to the metric $g(t)$. This result is an extension of the time-homogeneous case, in which the term containing the derivative $\partial_tg(t)$ is non-existent.\\

A result related to this is Schilder's theorem, which is proven in the time-homogeneous Riemannian setting in \cite{KRV18}. Schilder's theorem is concerned with the large deviations for Brownian paths when the variance tends to 0. More precisely, on the exponential scale we have
$$
\PP\left(X_t^\epsilon \approx \gamma\right) \approx \exp\left\{-\frac1{2\epsilon}\int_0^1|\dot\gamma(t)|^2\ud t\right\},
$$   
where $X_t^\epsilon = X_{\epsilon t}$. Our aim is to extend this result to the context of a manifold with a time-dependent metric. We follow the approach taken in \cite[Section 6]{KRV18}. For this, we define an appropriate way of lifting a Brownian motion with respect to a time-dependent metric to the frame bundle over the manifold, obtaining a so called horizontal Brownian motion. We then embed the frame bundle into Euclidean space, and use Freidlin-Wentzell theory to prove large deviations for embedded horizontal Brownian motion. Finally, we apply the contraction principle to obtain the large deviations for the Brownian motion with respect to a time-dependent metric in the manifold.\\

The paper is organized as follows. In Section \ref{section:main_result} we introduce the large deviation principle and fix the notation from Riemannian geometry. Furthermore, we introduce the notion of a Brownian motion with respect to a time-dependent metric and state the main result, the analogue of Schilder's theorem. Additionally, we sketch the approach to proving this result. Section \ref{section:lift_development} is devoted to the required theory of bundles and horizontal lifting of curves to such bundles. In particular, we define these notions with respect to a time-dependent metric. Finally, in Section \ref{section:proof} we provide all details of the proof of our main result.


\section{Main result} \label{section:main_result}

In this section we define the notion of a large deviation principle. Furthermore, we fix the notation from Riemannian geometry. Additionally, following \cite{ACT08,Cou11}, we define Brownian motion with respect to a collection of metrics $\{g(t)\}_{t\in[0,1]}$. Finally, we state our main result, and give an overview on how we will prove the result in Section \ref{section:proof}.

\subsection{Large deviation principle}

In large deviations one deals with the limiting behaviour of a sequence of random variables $\{X_\epsilon\}_{\epsilon>0}$. More precisely, large deviations quantify this limiting behaviour on an exponential scale by means of a rate function. We have the following definition.

\begin{definition}
Let $\{X_\epsilon\}_{\epsilon>0}$ be a sequence of random variables taking values in a metric space $\Xx$.
\begin{enumerate}
\item A \emph{rate function} is a lower semicontinuous function $I: \Xx \to [0,\infty]$. A rate function is called \emph{good} if its level sets $\{x \in \Xx| I(x) \leq \alpha\}$ are compact for any $\alpha \geq 0$.
\item The sequence $\{X_\epsilon\}_{\epsilon>0}$ satisfies the \emph{large deviation principle (LDP)} in $\Xx$ with rate function $I$ if the following hold:
\begin{enumerate}
\item (Upper bound) For any $F \subset \Xx$ closed
\begin{equation}\label{eq:upper_bound}
\limsup_{\epsilon\to0} \epsilon\log\PP(X_\epsilon \in F) \leq -\inf_{x\in F} I(x).
\end{equation}
\item (Lower bound) For any $G \subset \Xx$ open
$$
\liminf_{\epsilon\to0} \epsilon\log\PP(X_\epsilon \in G) \geq -\inf_{x\in G} I(x).
$$
\end{enumerate}
\end{enumerate}
\end{definition}

When a sequence $\{X_\epsilon\}_{\epsilon>0}$ satisfies the large deviation principle with rate function $I$, it is often informally written as
$$
\PP(X_\epsilon \approx x) \approx e^{-\epsilon^{-1}I(x)}.
$$

\subsection{$g(t)$-Brownian motion and the main result}

Let $M$ be a manifold, which in our case always means it is smooth and second countable. As usual, we denote by $TM$ the tangent bundle and for $x \in M$ we write $T_xM$ for the tangent space at $x$. Smooth sections of $TM$ are referred to as vector fields, and the collection of all vector fields is denoted by $\Gamma(TM)$. 

Let $\Gg = \{g(t)\}_{t\in[0,1]}$ be a collection of Riemannian metrics on $M$, smoothly depending on $t$. We will interchangeably use $\Gg$ and $\{g(t)\}_{t\in[0,1]}$ to refer to this collection of metrics. For $x \in M$ and $v,w \in T_xM$ we write $\inp{v}{w}_{g(t)}$ for the inner product of $v$ and $w$ with respect to the metric $g(t)$. For every $t \in [0,1]$, we denote by $\Nabla^t$ the Levi-Civita connection of $g(t)$, and by $\Delta_M^t$ the associated Laplace-Beltrami operator.

We denote by $C(M)$ the set of continuous functions on $M$, by $C_b(M)$ the set of bounded, continuous functions and by $C^\infty(M)$ the set of smooth functions. Furthermore, the set of continuous curves (on $[0,1]$) in $M$ is denoted by $C([0,1];M)$ and the set of smooth curves by $C^\infty([0,1];M)$. Additionally,  we define the space $H^1([0,1];M)$ by
$$
H^1([0,1];M) = \left\{\gamma:[0,1]\to M \middle | \gamma \mbox{ is differentiable a.e. and } \int_0^1 |\dot\gamma(t)|_{g(t)}^2\ud t < \infty\right\}.
$$
A subscript indicates that we only consider curves with that given initial value, i.e. we write $C_x([0,1];M), C_x^\infty([0,1];M)$ and $H_x^1([0,1];M)$ if we only consider curves with initial value $x \in M$.\\

We now define what we mean by a Brownian motion with respect to a collection of metrics $\{g(t)\}_{t\in[0,1]}$. We follow the definition in \cite{Cou11} which is equivalent to the definition in \cite{ACT08}.

\begin{definition}
Let $M$ be a manifold, and let $\{g(t)\}_{t\in[0,1]}$ be a collection of Riemannian metrics on $M$, smoothly depending on $t$. A process $X_t$ is called a $g(t)$-\emph{Brownian motion} if it is continuous and if for all $f \in C^\infty(M)$,
$$
f(X_t) - f(X_0) - \frac12\int_0^t \Delta_M^sf(X_s)\ud s
$$
is a local martingale. In that case, we say that $X_t$ is \emph{generated} by (the time-dependent generator) $\Delta_M^t$. 
\end{definition}

In general, a $g(t)$-Brownian motion only exists up to some explosion time $e(X)$. In the time-homogeneous setting we have that if the Ricci-curvature is bounded from below, then $e(X)$ is almost surely infinite, see e.g. \cite[Section 4.2]{Hsu02}. This result is extended to the time-inhomogeneous case in \cite{KP11} by requiring that $g(t)$ evolves under the backwards super Ricci flow, i.e., $g(t)$ satisfies
$$
\partial_1g(t) \leq \Ric_{g(t)}.
$$
In that case, $g(t)$-Brownian motion exists up to time $T$ for every $T > 0$.\\

Next, we state the main result, which is the analogue of Schilder's theorem for a $g(t)$-Brownian motion. Before we do this, we first need to introduce a proper rescaling of a $g(t)$-Brownian motion. 

To motivate the rescaling, first consider a standard real-valued Brownian motion $W_t$. Then Schilder's theorem states that
$$
\PP\left((\sqrt\epsilon W_t)_{0\leq t \leq 1} \approx \gamma\right) \approx \exp\left\{-\frac{1}{2\epsilon}\int_0^1 |\dot\gamma(t)|^2\ud t\right\}.
$$
Since $W_t$ is generated by $\frac12\Delta$, the process $\sqrt{\epsilon}W_t$ is generated by $\frac{\epsilon}{2}\Delta$. To extend this to the Riemannian setting, note that $\sqrt\epsilon W_t = W_{\epsilon t}$. As proven in \cite{KRV18}, if $X_t$ is a Riemannian Brownian motion, then
$$
\PP\left((X_{\epsilon t})_{0\leq t \leq 1} \approx \gamma\right) \approx \exp\left\{-\frac{1}{2\epsilon}\int_0^1 |\dot\gamma(t)|_g^2\ud t\right\},
$$
where the process $X_{\epsilon t}$ is generated by $\frac\epsilon2\Delta_M$. 

In the time-inhomogeneous setting, we want the process $X_{\epsilon t}$ to evolve according to a collection of metrics $\{g(t)\}_{t\in[0,1]}$. Consequently, we have to consider $X_t$ as a $g(\epsilon^{-1}t)$-Brownian motion, i.e., $X_t$ is generated by $\frac12\Delta_M^{\epsilon^{-1}t}$. In that case, substitution yields that the process $X_{\epsilon t}$ is generated by $\frac\epsilon 2\Delta_M^{t\epsilon \epsilon^{-1}} = \frac{\epsilon}{2}\Delta_M^t$. Our main result gives the large deviations for the process $X_t^\epsilon = X_{\epsilon t}$. 

Before we give the precise statement, we first give some motivation by considering the one-dimensional, real-valued case. For this, let $g:[0,1] \to \RR_{>0}$ be a smooth function with associated inner products given by $\inp{v}{w}_{g(u)} = g(u)vw$. Let $W_t^g$ be the process generated by $\frac{1}{2g(t)}\frac{\dd^2}{\dd x^2}$, i.e., $W_t^g$ is a $g(t)$-Brownian motion. Then the rescaled process $W_{\epsilon t}^{g(\epsilon^{-1}\cdot)}$ is generated by $\frac{\epsilon}{2g(\epsilon^{-1}\epsilon t)}\frac{\dd^2}{\dd x^2} = \frac{\epsilon}{2g(t)}\frac{\dd^2}{\dd x^2}$. If we denote by $W_t$ a standard Brownian motion, we have 
$$
W_t^g = W_{\int_0^t\frac{1}{g(r)}\ud r}.
$$
Consequently,
$$
W_{\epsilon t}^{g(\epsilon^{-1}\cdot)} = W_{\int_0^{\epsilon t} \frac{1}{g(\epsilon^{-1}r)}\ud r} = W_{\epsilon\int_0^t\frac{1}{g(r)}\ud r} = \sqrt{\epsilon}W_{\int_0^t\frac{1}{g(r)}\ud r}.
$$ 
If we write $\psi(t) = \int_0^t \frac{1}{g(r)}\ud r$, we have
$$
\PP\left(W_{\epsilon \cdot}^{g(\epsilon^{-1}\cdot)} \approx \gamma\right) = \PP\left(\sqrt{\epsilon}W_{\psi(\cdot)} \approx \gamma\right) = \PP\left(\sqrt{\epsilon}W_{\cdot} \approx \gamma(\psi^{-1}(\cdot))\right).
$$
Hence, Schilder's theorem implies that
$$
\PP\left(W_{\epsilon \cdot}^{g(\epsilon^{-1}\cdot)} \approx \gamma\right) \approx \exp\left\{-\frac{1}{2\epsilon}\int_0^{\psi(1)} \left(\frac{\dd}{\dd t}(\gamma(\psi^{-1}(t))\right)^2\ud t\right\}.
$$
Now by the inverse function theorem, $(\psi^{-1})'(t) = \frac{1}{\psi'(\psi^{-1}(t))} = g(\psi^{-1}(t))$. Consequently, we have
\begin{align*}
\int_0^{\psi(1)} \left(\frac{\dd}{\dd t}(\gamma(\psi^{-1}(t))\right)^2\ud t
&=
\int_0^{\psi(1)} \dot\gamma(\psi^{-1}(t))^2 g(\psi^{-1}(t))^2\ud t
\\
&=
\int_0^1 \dot \gamma(u)^2 g(u) \ud u
\\
&=
\int_0^1 |\dot\gamma(u)|_{g(u)}^2\ud u.
\end{align*}
Here we used in the second line the substitution $u = \psi^{-1}(t)$. 

Collecting everything, we have
$$
\PP\left(W_{\epsilon t}^{g(\epsilon^{-1}\cdot)} \approx \gamma\right) \approx \exp\left\{-\frac{1}{2\epsilon}\int_0^1 |\dot\gamma(u)|_{g(u)}^2\ud u\right\}.
$$

Our main theorem states that this happens in general.

\begin{theorem}\label{theorem:Schilder_time}
Let $M$ be a Riemannian manifold and let $\{g(t)\}_{t\in[0,1]}$ be a collection of Riemannian metrics, smoothly depending on $t$. Fix $x_0 \in M$, and let $X_t$ be a $g(t)$-Brownian motion with $X_0 = x_0$. Assume $X_t$ exists for all time $t \in [0,1]$. Furthermore assume that for every $\epsilon > 0$, the continuous process $X_t^\epsilon$ generated by $\frac\epsilon2\Delta_M^t$ exists for all time $t \in [0,1]$. Then $\{X_t^\epsilon\}_{\epsilon > 0}$ satisfies the large deviation principle in $C([0,1];M)$ with good rate function $I_M$ given by
\begin{equation}\label{eq:rate_function_timeBM}
I_M(\gamma) =
\begin{cases}
\frac12\int_0^1 |\dot\gamma(t)|_{g(t)}^2\ud t,		& \gamma \in H_{x_0}^1([0,1];M),\\
\infty									& \mbox{otherwise.}
\end{cases}
\end{equation}
\end{theorem}


\subsection{Sketch of the proof Theorem \ref{theorem:Schilder_time}}\label{subsection:sketch_proof}

The proof of Theorem \ref{theorem:Schilder_time} follows the same lines as the proof given in \cite[Section 6]{KRV18} for the time-homogeneous case. The main work lies in defining a good analogue of the concept of horizontal lift and anti-development in the time-inhomogeneous case. The detailed construction is given in Section \ref{section:lift_development}.  

Instead of proving the large deviation principle for $X_t^\epsilon$ directly, we first prove the large deviation principle for its horizontal lift $U_t^\epsilon$ with respect to $\{g(t)\}_{t\in[0,1]}$ in the frame bundle $FM$. As explained in Section \ref{subsection:lift_BM} (see also \cite{Cou11,ACT08}), this process satisfies the Stratonovich stochastic differential equation
\begin{equation}\label{eq:horizontal_sketch_proof}
\dd U_t^\epsilon = H_i(t,U_t^\epsilon)\circ \dd W_t^{\epsilon,i} - \frac12(\partial_tg(t))_{ij}(U_t^\epsilon e_i,U_t^\epsilon e_j)V^{ij}(U_t^\epsilon)\ud t.
\end{equation}
Here, $H_i(t,\cdot)$ are the fundamental horizontal fields with respect to the metric $g(t)$ and $V^{ij}$ is the canonical basis of vertical vector fields over $FM$, see Section \ref{subsection:horizontal_lift}. Furthermore, $\{e_1,\ldots,e_d\}$ denotes the standard basis of $\RR^d$.

By embedding $FM$ smoothly in some Euclidean space $\RR^N$, we can push-forward the equation \eqref{eq:horizontal_sketch_proof} to $\RR^N$ to obtain a stochastic differential equation on $\RR^N$ with a drift, and a diffusion of order $\sqrt{\epsilon}$, see e.g. \cite[Section 1.2]{Hsu02}. Consequently, at least if we restrict to compact sets, we can apply Theorem \ref{theorem:time_inhomogeneous_FW} in Section \ref{section:FW_theory} to get the large deviations for the embedded process. By the contraction principle (see \cite[Theorem 4.2.1]{DZ98}, this can then be transferred to the process $X_t^\epsilon$. The relation between the derivative of a curve in $M$ and the derivative of its anti-development with respect to $\{g(t)\}_{t\in[0,1]}$ in $\RR^d$ then assures that we obtain the correct rate function.

Finally, as shown in Section \ref{subsection:compact_containment}, we can use a general approach using Lyapunov functions to show that the process $X_t^\epsilon$ remains in a compact set with high probability. This, together with the result obtained when restricting to compact sets, allows us to obtain the full result of Theorem \ref{theorem:Schilder_time}.


\section{Horizontal lift and anti-development}\label{section:lift_development}

In this section we discuss how to define a horizontal lift with respect to a collection $\{g(t)\}_{t\in[0,1]}$ of metrics on a manifold $M$. In order to do this, we need a suitable definition of what we mean by horizontal curves and horizontal vectors. For this, we need to incorporate time into our analysis. In order for the upcoming constructions to make sense also for $t \notin [0,1]$, we set $g(t) = g(0)$ for $t < 0$ and $g(t) = g(1)$ for $t > 1$. 

\subsection{A time-dependent connection which is metric}

Denote spacetime by $\MM := \RR \times M$ and let $T\MM$ be its tangent bundle. For $(t,x) \in \MM$ we have $T_{(t,x)}\MM = \RR \oplus T_xM$. We denote the basis tangent vector in the time-direction by $\partial_1$.

Instead of considering the tangent bundle $T\MM$, we also want to view $TM$ as bundle over $\MM$. More precisely, we define the bundle $\overline{TM}$ over $\MM$ with fibres given by
$$
\overline{TM}_{(t,x)} = T_xM
$$
for all $t \in \RR$ and all $x \in M$. A smooth section of $\overline{TM}$ is called a time-dependent vector field. We will sometimes write $Z(t) \in \Gamma(\overline{TM})$ to stress that $Z$ is a time-dependent vector fields on $M$.

To define the desired connection on $\overline{TM}$, we first need to define what we mean by the derivative of $g(t)$ with respect to $t$. This is a 2-tensor $\partial_1g(t):TM \times TM \to C^\infty(M)$, which in coordinates is given by
$$
\partial_1g(t)(v,w) = \partial_1g_{ij}(t)v^iw^j,
$$
where $v = v^i\partial_i$, $w = w^j\partial_j$ and $g(t) = g_{ij}(t)\dd x^i\otimes \dd x^j$. 

Furthermore, for $Y \in \Gamma(TM)$, we denote by $(\partial_1g(t))(Y(t),\cdot)^{\#_t}$ the vector field obtained by 'raising an index' with respect to the metric $g(t)$. More precisely, it is the unique vector field such that for all vector fields $Z \in \Gamma(TM)$ we have
$$
(\partial_1g(t))(Y(t),Z) = \inp{(\partial_1g(t))(Y(t),\cdot)^{\#_t}}{Z}_{g(t)}.
$$

Finally, we denote by $\Nabla^t$ the Levi-Civita connection of the metric $g(t)$.\\ 

Following the idea in \cite{Ham86,Ham93}, see also Chapter 6 in \cite{AH11}, we equip the bundle $\overline{TM}$ over $\MM$ with a natural connection $\Nabla: \Gamma(T\MM) \times \Gamma(\overline{TM}) \to \Gamma(\overline{TM})$ given by
\begin{equation}\label{eq:connection_spacetime}
\begin{cases}
\Nabla_XY(t) = \Nabla_X^tY,\\
\Nabla_{\partial_1}Y(t) = \partial_1Y(t) + \frac12(\partial_1g(t))(Y(t),\cdot)^{\#_t},
\end{cases}
\end{equation}
for $X \in \Gamma(TM)$ a vector field over $M$ and $Y \in \Gamma(\overline{TM})$ a time-dependent vector field over $M$. By $C^\infty$-linearity, this defines $\Nabla_ZY$ for all $Z \in \Gamma(T\MM)$ and all $Y \in \Gamma(\overline{TM})$. This connection is compatible with the collection $\{g(t)\}_{t\in[0,1]}$ of Riemannian metrics on $M$, as we will show in the following proposition.

\begin{proposition}\label{prop:metric_connection}
The connection defined in \eqref{eq:connection_spacetime} is metric in the following sense: for all time-dependent vector fields $X,Y \in \Gamma(\overline{TM})$ and $Z \in \Gamma(T\MM)$ we have 
$$
Z\inp{X(t)}{Y(t)}_{g(t)}
= 
\inp{\Nabla_ZX(t)}{Y(t)}_{g(t)} + \inp{X(t)}{\Nabla_ZY(t)}_{g(t)}
$$
for all $t \in \RR$. 
\end{proposition}
\begin{proof}
Note that $Z \in \Gamma(T\MM)$ can be written as $Z(t,x) = c_1(t,x)\partial_t + \tilde Z(t)(x)$ where $c_1:\MM \to \RR$ is a smooth function and $\tilde Z \in \Gamma(\overline{TM})$ a time-dependent vector field over $M$. Since $\Nabla^t$ is metric with respect to $g(t)$, we have
\begin{align*}
\tilde Z(t)\inp{X(t)}{Y(t)}_{g(t)} 
&=
\inp{\Nabla^t_{\tilde Z(t)}X(t)}{Y(t)}_{g(t)} + \inp{X(t)}{\Nabla^t_{\tilde Z(t)}Y(t)}_{g(t)}
\\
&=
\inp{\Nabla_{\tilde Z(t)}X(t)}{Y(t)}_{g(t)} + \inp{X(t)}{\Nabla_{\tilde Z(t)}Y(t)}_{g(t)}.
\end{align*}

 For the derivative with respect to $\partial_1$, if we write $X(t) = X^i(t)\partial_i$, $Y(t) = Y^j(t)\partial_j$ and $g(t) = g_{ij}(t)\dd x^i\otimes \dd x^j$ in coordinates, we get
\begin{align*}
\partial_1\inp{X(t)}{Y(t)}_{g(t)}
&=
\partial_1(X^i(t)Y^j(t)g_{ij}(t))
\\
&=
\partial_1X^i(t)Y^j(t)g_{ij}(t) + X^i(t)\partial_1 Y^j(t)g_{ij}(t) + X^i(t)Y^j(t)\partial_1 g_{ij}(t)
\\
&=
\inp{\partial_1X(t)}{Y(t)}_{g(t)} + \inp{X(t)}{\partial_1Y(t)}_{g(t)} + (\partial_1g(t))(X(t),Y(t))
\\
&=
\inp{\Nabla_{\partial_1}X(t)}{Y(t)}_{g(t)} + \inp{X(t)}{\Nabla_{\partial_1}Y(t)}_{g(t)}.
\end{align*}
Here, the last line follows by splitting $(\partial_1g(t))(X(t),Y(t))$ in two, and raising one index. 

Finally, using that $\Nabla$ is $C^\infty$-linear in the first variable proves the claim.
\end{proof}

As a corollary, we obtain the derivative of the inner product between two time-dependent vector fields along a curve in $M$.

\begin{corollary}\label{cor:metric_connection}
Let $X(t),Y(t) \in \Gamma(\overline{TM})$ be time-dependent vector fields and let $\gamma:[0,1] \to M$ be a curve. Then
\begin{align*}
&\frac{\dd}{\dd t}\inp{X(t,\gamma(t))}{Y(t,\gamma(t))}_{g(t)}
\\
&=
\inp{\Nabla_{\partial_1 + \dot\gamma(t)}X(t)}{Y(t)}_{g(t)} + \inp{X(t)}{\Nabla_{\partial_1 + \dot\gamma(t)}Y(t)}_{g(t)}.
\end{align*}
\end{corollary}
\begin{proof}
Consider the curve $\varphi:[0,1] \to \MM$ given by $\varphi(t) = (t,\gamma(t))$. From Proposition \ref{prop:metric_connection} it follows that
\begin{align*}
&\frac{\dd}{\dd t}\inp{X(t,\gamma(t))}{Y(t,\gamma(t))}_{g(t)}
\\
&=
\dot\varphi(t)\inp{X(t,\gamma(t))}{Y(t,\gamma(t))}_{g(t)}
\\
&=
\inp{\Nabla_{\dot\varphi(t)}X(t,\gamma(t))}{Y(t,\gamma(t))}_{g(t)} + \inp{X(t,\gamma(t))}{\Nabla_{\dot\varphi(t)}Y(t,\gamma(t))}_{g(t)}
\\
&=
\inp{\Nabla_{\partial_1 + \dot\gamma(t)}X(t)}{Y(t)}_{g(t)} + \inp{X(t)}{\Nabla_{\partial_1 + \dot\gamma(t)}Y(t)}_{g(t)}.
\end{align*}
Here, the last line follows from the fact that $\dot\varphi(t) = \partial_1 + \dot\gamma(t)$.
\end{proof}

\begin{remark}
If $X(t) = X$ for some fixed vector field $X \in \Gamma(TM)$, then $\partial_1X(t) = 0$, and we reduce to the setting in \cite{Cou11}. If we consider another stationary vector field $Y(t) = Y \in \Gamma(TM)$ and a curve $\gamma:[0,1] \to M$, we have
\begin{multline}
\frac{\dd}{\dd t}\inp{X(\gamma(t))}{Y(\gamma(t))}_{g(t)} = (\partial_1g(t))(X(\gamma(t)),Y(\gamma(t)))~+\\
+ \inp{\Nabla_{\dot\gamma(t)}^tX(\gamma(t))}{Y(\gamma(t))}_{g(t)} + \inp{X(\gamma(t))}{\Nabla_{\dot\gamma(t)}^tY(\gamma(t))}_{g(t)}.
\end{multline}
\end{remark}

Corollary \ref{cor:metric_connection} inspires us to define a notion of a time-dependent vector field being parallel along a curve in $M$ with respect to a collection $\{g(t)\}_{t\in[0,1]}$ of Riemannian metrics. We have the following definition.

\begin{definition}\label{def:time_parallel_transport}
Let $\gamma:[0,1] \to M$ be a curve. A time-dependent vector field $X(t)$ along $\gamma$ is said to be \emph{parallel along} $\gamma$ \emph{with respect to} $\{g(t)\}_{t\in[0,1]}$ if it is parallel along the curve $(t,\gamma(t))$ in $\MM$ with respect to the connection $\Nabla$. More precisely, $X(t)$ is parallel along $\gamma$ if and only if for all $t\in[0,1]$ we have
$$
\Nabla_{\partial_1 + \dot\gamma(t)}X(t)(\gamma(t)) = 0.
$$
\end{definition}

\begin{remark}
If $X(t)$ and $Y(t)$ are time-dependent vector fields which are parallel along $\gamma$ with respect to $\{g(t)\}_{t\in[0,1]}$, then by Corollary \ref{cor:metric_connection} we have
$$
\frac{\dd}{\dd t} \inp{X(\gamma(t))}{Y(\gamma(t))}_{g(t)} = 0.
$$
This shows that the inner product between parallel vector fields is constant. In particular, by taking $Y(t) = X(t)$, we find that $|X(t,\gamma(t))|_{g(t)}$ is constant along $\gamma(t)$.
\end{remark}

\subsection{Horizontal lift}\label{subsection:horizontal_lift}

The frame bundle $FM$ over $M$ is the bundle with fibres given by 
$$
F_xM = \{u:\RR^d \to T_xM| u \mbox{ linear isomorphism}\}
$$
The frame bundle is a principal bundle with structure group $GL(d,\RR)$, the set of all invertible $d\times d$ matrices. It is a manifold, with the projection $\pi:FM \to M$ being a smooth map. Furthermore, the tangent bundle of $FM$ can be split in two parts, namely in directions in $M$ (defining a connection on $M$) and in the direction of the frames, i.e., vectors tangent to the fibres of $FM$. 
More precsiely, if $V \in T_uFM$ is tangent to the fibre $F_{\pi u}M$, then $V$ is said to be \emph{vertical}. Consequently, $V \in T_uFM$ is vertical if and only if it is the tangent vector of a curve in $FM$ that remains inside $F_{\pi u}M$. We denote the vertical subspace of $T_uFM$ by $V_uFM$. If we consider the map $L_u: GL(d,\RR) \to F_{\pi u}M$ given by $L_ug = ug$, then a basis of $V_uFM$ is given by
\begin{equation}\label{eq:vertical_basis}
V_{ij}(u) = \dd L_u(I)(E_{ij}),
\end{equation}
where $E_{ij}$ is the matrix of all zeros, except for a 1 in position $(i,j)$.\\

A subspace $H_u$ of $T_uFM$ such that $T_uFM = H_uFM \oplus V_uFM$ is called a horizontal subspace. Choosing a collection $\{H_u\}_{u\in FM}$ of horizontal subspaces smoothly depending on $u$ is equivalent to choosing a connection on $M$, see e.g. \cite[Chapter 8]{Spi79vol2} or \cite[Chapter 2]{KN63}.

We show how to obtain the collection of subspaces $\{H_u\}_{u\in FM}$ when $M$ is equipped with a connection $\Nabla$.  First, a \emph{horizontal lift of a curve} $\gamma:[0,1] \to M$ is a curve $u:[0,1] \to FM$ such that for all $a \in \RR^d$ we have
$$
\Nabla_{\dot\gamma(s)}(u(s)a) = 0
$$
for all $s \in [0,1]$. This horizontal lift exists for all time, and is unique once we fix the initial frame $u(0) = u$, see e.g. \cite{Spi79vol2,KN63}. Intuitively, the horizontal lift of $\gamma$ can be thought of as a parallel selection of frames along $\gamma$. 

To define the \emph{horizontal lift of a vector} $X \in T_pM$ via some frame $u \in F_pM$, consider a curve $\gamma$ with $\gamma(0) = p$ and $\dot\gamma(0) = X$. If $u(s)$ is the horizontal lift of $\gamma$ with $u(0) = u$, then we define the horizontal lift $X^*(u) = \dot u(0)$. With the horizontal lift at hand, for $u \in FM$, we set
$$
H_u = \{X^*(u)| X \in T_{\pi u}M\}.
$$ 

Finally, given $a \in \RR^d$ and $u \in F_pM$, we have that $ua \in T_pM$, so that we can define its horizontal lift. We denote this by $H(u)a$, which is thus given by
\begin{equation}\label{eq:horizontal_lift_frame_vector}
H(u)a = (ua)^*(u).
\end{equation}

\subsubsection{Horizontal lift with respect to a family of metrics}

Instead of performing horizontal lift with respect to a fixed connection, we wish to define it with respect to a time-dependent family of connections. More precisely, we wish to define the horizontal lift with respect to the family of Levi-Civita connections associated to the collection $\Gg = \{g(t)\}_{t\in[0,1]}$ of metrics on $M$. To do this, we use the parallel transport given in Definition \ref{def:time_parallel_transport}. 

\begin{definition}\label{def:horizontal_lift}
Let $\gamma:[0,1] \to M$ be a curve in $M$. A curve $u(t) \in FM$ is called a \emph{horizontal lift of $\gamma$ with respect to} $\{g(t)\}_{t\in[0,1]}$ if for all $a \in \RR^d$ we have that $u(t)a$ is parallel along $\gamma$ with respect to $\{g(t)\}_{t\in [0,1]}$, i.e., for all $a \in \RR^d$ we have
$$
\Nabla_{\partial_1 + \dot\gamma(t)}(u(t)a) = 0
$$
for all $t \in [0,1]$.
\end{definition}  

If $u(t)$ is the horizontal lift with respect to $\{g(t)\}_{t\in[0,1]}$ of a curve $\gamma$, then by Corollary \ref{cor:metric_connection} we have for all $a \in \RR^d$ that
\begin{equation}\label{eq:derivative_distance_lift}
\frac{\dd}{\dd t}|u(t)a|_{g(t)} = 0, 
\end{equation}
i.e., $|u(t)a|_{g(t)} = |u(0)a|_{g(0)}$ for all $t \in [0,1]$. Consequently, if $u(0):\RR^d \to (T_{\gamma(0)}M,g(0))$ is an isometry, then $u(t):\RR^d \to (T_{\gamma(t)}M,g(t))$ is an isometry for all $t \in [0,1]$.

 We use this observation to show that the horizontal lift with respect to $\{g(t)\}_{t\in[0,1]}$ of metrics exists for all time, and is unique once an initial (orthonormal) frame is given. We do this by showing that the horizontal lift defined in \ref{def:horizontal_lift} is a special instance of a horizontal lift from the manifold $\MM = \RR \times M$ to a principle fibre bundle over $\MM$. To this end, consider the bundle $\Oo$ over $\MM$ with fibres given by
\begin{equation}\label{eq:set_orthonormal_frames_t}
\Oo_{(t,x)} = \{u:\RR^d \to (T_xM,g(t))| u \mbox{ isometry}\},
\end{equation}
i.e., $\Oo_{(t,x)}$ consists of the orthonormal frames for $T_xM$ with respect to the metric $g(t)$. 

The bundle $\Oo$ is a principal bundle with structure group $G = O(d)$, the orthogonal group. Now, let $\gamma(t)$ be a curve in $M$ with horizontal lift $u(t)$ as in Definition \ref{def:horizontal_lift}, such that $u(0)$ is an orthonormal frame for $T_{\gamma(0)}M$ with respect to $g(0)$. From \eqref{eq:derivative_distance_lift} it follows that for all $t \in [0,1]$, $u(t)$ is orthonormal with respect to $g(t)$, and hence $u(t) \in \Oo_{(t,\gamma(t))}$ for all $t \in [0,1]$. If we now define $\varphi(t) = (t,\gamma(t)) \in \MM$, then $\dot\varphi(t) = \partial_1 + \dot\gamma(t)$. Putting everything together, this implies that the curve $u(t)$ can also be interpreted as the horizontal lift of $\varphi(t)$ with respect to the connection $\Nabla$ as in \eqref{eq:connection_spacetime} to the bundle $\Oo$. Because $\Oo$ is a principal bundle, it follows that a horizontal lift of $\varphi(t) = (t,\gamma(t))$ exists for all time $t \in [0,1]$ and is unique if the initial condition $u(0) = u_0 \in \Oo_{(0,\gamma(0))}$ is fixed. For this, we refer to (among others) \cite[Chapter 8]{Spi79vol2}. This implies that the horizontal lift defined in Definition \ref{def:horizontal_lift} always exists, and is unique if an initial orthonormal frame with respect to $g(0)$ is given.\\

As explained in the previous section, if a horizontal lift for curves is defined, we can use this to define the horizontal lift of tangent vectors. In particular, the horizontal lift of curves in $\MM$ to the bundle $\Oo$ with respect to the connection $\Nabla$ in \eqref{eq:connection_spacetime} allows us to lift tangent vectors $X \in T_{(t,x)}\MM$ to $T\Oo$. In what follows, we denote this lift by $X^*$. 

Since we also have a notion of horizontal lifts of curves in $M$ with respect to $\{g(t)\}_{t\in[0,1]}$, we can use this to define the horizontal lift of a tangent vector in $TM$ with respect to $\{g(t)\}_{t\in[0,1]}$.

\begin{definition}\label{definition:lift_vector_family_metrics}
Let $X \in T_pM$ and $u \in \Oo_{(s,p)}$. Let $\gamma:[0,1] \to M$ be a curve with $\gamma(s) = p$ and $\dot\gamma(s) = X$. Denote by $u(t)$ the horizontal lift of $\gamma$ with respect to $\Gg = \{g(t)\}_{t\in[0,1]}$, satisfying $u(s) = u$. We define the \emph{horizontal lift of} $X$ \emph{via} $u$ \emph{with respect to} $\{g(t)\}_{t\in[0,1]}$ by $X^{*\Gg}(u) = \dot u(s)$.
\end{definition}

\begin{remark}\label{remark:equality_horizontal_lifts}
If $\gamma$ is a curve in $M$, we can identify its horizontal lift with respect to $\{g(t)\}_{t\in[0,1]}$ with the horizontal lift of the curve $\varphi(t) = (t,\gamma(t))$ in $\MM$ with respect to the connection $\Nabla$ defined in \eqref{eq:connection_spacetime}. This implies that $\dot u(s)$ is the horizontal lift of $\dot\varphi(s) = \partial_1 + \dot\gamma(s)$ to $T_{u(s)}\Oo_{(s,\gamma(s))}$ via $u(s)$. Consequently, we have that $X^{*\Gg}(u) = (\partial_1 + X)^*(u)$.  
\end{remark}



We now wish to relate the horizontal lift of $X$ via $u \in \Oo_{(s,p)}$, with respect to $\{g(t)\}_{t\in[0,1]}$ to the horizontal lift of $X$ via $u$ with respect to the metric $g(s)$. Before we get to this, we first need the following result, the proof of which is inspired by the proof of \cite[Proposition 1.2]{Cou11}.

\begin{proposition}\label{prop:lift_partial_t}
Let $u \in \Oo_{(s,p)}$. Then the horizontal lift of $\partial_1$ via $u$ with respect to the connection $\Nabla$ in \eqref{eq:connection_spacetime} is given by
$$
\partial_1^*(u) = -\frac12(\partial_1g(s))(ue_i,ue_j)V_{ij}(u).
$$
Here, $\{e_1,\ldots,e_d\}$ is the canonical basis of $\RR^d$ and $V_{ij}(u)$ are the canonical vertical basis vectors of $V_uFM$ defined in \eqref{eq:vertical_basis}. 
\end{proposition}
\begin{proof}
Consider the curve $\eta(t) = (s+t,p)$. Then $\dot\eta(0) = \partial_1$. Let $u(t)$ be the horizontal lift of $\eta(t)$ with $u(0) = u$. Then $\partial_1^*(u) = \dot u(0)$. Since $\dot\eta(t) = \partial_1$, we have for all $a \in \RR^d$ that
$$
\Nabla_{\partial_1}(v(t)a) = 0,
$$
which gives via \eqref{eq:connection_spacetime} that
\begin{equation}\label{eq:derivative_lift_partial_t}
\partial_1(u(t)a) + \frac12(\partial_1g(s+t))(u(t)a,\cdot)^{\#_{s+t}} = 0.
\end{equation}

Since $u(t) \in F_pM$ for all $t$, we have that $\dot u(t) \in V_{u(t)}FM$. Consequently, we can write
$$
\dot u(t) = c_{\alpha\beta}(t,u(t))V_{\alpha\beta}(u(t)),
$$
where $V_{\alpha\beta}$ are the canonical vertical basis vector fields defined in \eqref{eq:vertical_basis}. Note that $u(t)a = \ev_{a}(u(t))$, where $\ev_a:FM \to TM$ is evaluation in $a$. From this it follows that $\partial_1(u(t)a) = \dd(\ev_{a})(u(t))(\dot u(t))$. Furthermore, note that
\begin{align*}
\dd(\ev_{a})(u(t))(V_{\alpha\beta}(u(t))
&=
\dd(\ev_a\circ L_{u(t)})(I)(E_{\alpha\beta})
\\
&=
\left.\frac{\dd}{\dd s}\right|_{s=0} u(t)(I + sE_{\alpha\beta})a
\\
&=
u(t)(a_\beta e_\alpha),
\end{align*}
where we write $a = a_\beta e_\beta$.

By linearity, we find for every $i = 1,\ldots,d$ that 
\begin{align*}
\partial_1(u(t)e_i) 
&= 
c_{\alpha\beta}(t,u(t))\dd~\ev_{e_i}(u(t))(V_{\alpha\beta}(u(t)))
\\
&= 
c_{\alpha\beta}(t,u(t))u(t)\delta_{i\beta}e_\alpha
\\
&= 
 c_{\alpha i}(t,u(t)) u(t)e_\alpha.
\end{align*}

Furthermore, since $\partial_1(u(t)e_i) = -\frac12(\partial_1g(s+t))(u(t)e_i,\cdot)^{\#_{s+t}}$ by \eqref{eq:derivative_lift_partial_t}, we have
$$
\inp{\partial_1(u(t)e_i)}{u(t)e_j}_{g(s+t)} = -\frac12(\partial_1g(s+t))(u(t)e_i,u(t)e_j)
$$
for every $j = 1,\ldots,d$. Now, the left hand side is given by
\begin{align*}
\inp{\partial_1(u(t)e_i)}{u(t)e_j}_{g(s+t)} 
&= 
c_{\alpha i}(t,u(t))\inp{u(t)e_\alpha}{u(t)e_j}_{g(s+t)}
\\
&=
c_{\alpha i}(t,u(t))\inp{e_\alpha}{e_j}_{\RR^d}
\\
&=
c_{ji}(t,u(t)).
\end{align*}
Here we used in the second line that $u(t) \in \Oo_{(s+t,p)}$, to that it is an isometry from $\RR^d$ to $(T_pM,g(s+t))$.

Combining the two equalities above, we find for every $i,j = 1,\ldots,d$ that
$$
c_{ji}(t,u(t)) = -\frac12(\partial_1g(s+t))(u(t)e_i,u(t)e_j).
$$
Because $\partial_1g(s+t)$ is symmetric, it follows that $c_{ij} = c_{ji}$. Consequently, we can write
$$
\dot u(t) = -\frac12(\partial_1g(s+t))(u(t)e_i,u(t)e_j)V_{ij}(u(t)), 
$$
so that
$$
\partial_1^*(u) = \dot u(0) = -\frac12(\partial_1g(s))(ue_i,ue_j)V_{ij}(u),
$$
where we used that $u(0) = u$.
\end{proof}

From Proposition \ref{prop:lift_partial_t} we deduce the relation the horizontal lift of $X \in T_pM$ with respect to $\{g(t)\}_{t\in[0,1]}$ and with respect to the metric $g(s)$ at a specific time $s \in [0,1]$.

\begin{corollary}\label{cor:expression_lift_collection}
For $X \in T_pM$ and $u \in \Oo_{(s,p)}$ we have
$$
X^{*\Gg}(u) = X^{*_s}(u) - \frac12(\partial_1g(s))(ue_i,ue_j)V_{ij}(u),
$$
where $X^{*_s}(u)$ denotes the horizontal lift of $X$ via $u$, with respect to the metric $g(s)$, and the $e_i$ and $V_{ij}$ are as in Proposition \ref{prop:lift_partial_t}.
\end{corollary}
\begin{proof}
From Remark \ref{remark:equality_horizontal_lifts} it follows that $X^{*\Gg}(u) = (\partial_1 + X)^*(u)$. Since $(\partial_1 + X)^*(u) = \partial_1^*(u) + X^*(u)$ (see e.g. \cite{Spi79vol2}), it follows from Proposition \ref{prop:lift_partial_t} that we are done once we show that $X^*(u) = X^{*_s}(u)$. To see the latter, consider a curve $\gamma:(-\epsilon,\epsilon) \to M$ with $\gamma(0) = p$ and $\dot\gamma(0) = X$ and define $\varphi:(-\epsilon,\epsilon) \to \MM$ by $\varphi(t) = (s,\gamma(t))$. Then $\varphi(0) = (s,p)$ and $\dot\varphi(0) = X$. Let $u(t)$ be the horizontal lift of $\varphi$ with $u(0) = u$. Since $\dot\varphi(t) = X$, we have
$$
\Nabla_X^s(u(t)a) = 0
$$
for every $a \in \RR^d$. Consequently, $u(t)$ is the horizontal lift of $\gamma(t)$ with respect to $\Nabla^s$, i.e., the Levi-Civita connection of $g(s)$. It follows that $X^*(u) = X^{*_s}(u)$ as desired.  
\end{proof}

\subsection{Development and anti-development of curves}\label{subsection:anti_development}

The idea is now to use the notion of a horizontal lift to associate to a curve in $M$ a curve in $\RR^d$ and vice versa. We have the following definition.

\begin{definition}
Let $\gamma:[0,1] \to M$ be a curve in $M$ and let $u(t)$ be a horizontal lift of $\gamma$ with respect to $\{g(t)\}_{t\in[0,1]}$. We define the \emph{anti-development of} $\gamma$ \emph{with respect to} $\{g(t)\}_{t\in[0,1]}$ as the curve $w:[0,1] \to \RR^d$ given by
\begin{equation}\label{eq:anti_development}
w(t) = \int_0^t u(s)^{-1}\dot\gamma(s) \ud s.
\end{equation}
\end{definition}

If we fix a frame $u \in \Oo_{(0,\gamma(0))}$ (see \eqref{eq:set_orthonormal_frames_t}), we can speak about the anti-development of $\gamma$ via $u$ with respect to $\{g(t)\}_{t\in[0,1]}$, since in that case the horizontal lift with respect to $\{g(t)\}_{t\in[0,1]}$ satisfying $u(0) = u$ is unique.\\

If $w(t)$ is the anti-development of $\gamma(t)$ with respect to $\{g(t)\}_{t\in[0,1]}$ via the horizontal lift $u(t)$, then \eqref{eq:anti_development} implies that
$$
\dot w(t) = u(t)^{-1}\dot\gamma(t),
$$
which rewrites to
$$
\dot\gamma(t) = u(t)\dot w(t).
$$
Since both sides are elements of $T_{\gamma(t)}M$, we can consider their horizontal lifts with respect to the metric $g(t)$, which must be equal:
\begin{equation}\label{eq:expression_lift_antidevelopment}
H(t,u(t))\dot w(t) := (u(t)\dot w(t))^{*_t} = (\dot\gamma(t))^{*_t}.
\end{equation}
Here $H(t,u(t)$ is as defined in \eqref{eq:horizontal_lift_frame_vector}, but with respect to the Levi-Civita connection $\Nabla^t$ for the metric $g(t)$. Furthermore, since $u(t)$ is the horizontal lift of $\gamma$ with respect to $\{g(t)\}_{t\in[0,1]}$, we have that $\dot u(t) = \dot\gamma(t)^{*\Gg}$. Consequently, by applying Corollary \ref{cor:expression_lift_collection} and using \eqref{eq:expression_lift_antidevelopment} we obtain
\begin{align*}
\dot u(t) 
&=
\dot\gamma(t)^{*\Gg}
\\
&= 
(\dot\gamma(t))^{*_t} - \frac12(\partial_tg(t))(u(t)e_i,u(t)e_j)V^{ij}(u(t))
\\
&=
H(t,u(t))\dot w(t) - \frac12(\partial_tg(t))(u(t)e_i,u(t)e_j)V^{ij}(u(t)).
\end{align*}

We thus obtained a differential equation for the horizontal lift $u$ with respect to $\{g(t)\}_{t\in[0,1]}$ in terms of the anti-development $w$. This shows how to invert the operation of taking the anti-development of a curve. We make the following definition.
\begin{definition}\label{def:time_development}
Let $w:[0,1] \to \RR^d$ be a curve in $\RR^d$ and fix $u_0 \in \Oo_{(0,p)}$. Let $u:[0,1] \to FM$ be the solution of
\begin{equation}\label{eq:time_development}
\dot u(t) = H(t,u(t))\dot w(t) - \frac12(\partial_tg(t))(u(t)e_i,u(t)e_j)V^{ij}(u(t))
\end{equation}
with $u(0) = u_0$, where $H(t,u(t))$ is as defined in \eqref{eq:horizontal_lift_frame_vector} for the Levi-Civita connection $\Nabla^t$ of the metric $g(t)$. Then the curve $\gamma(t) = \pi u(t)$ is called the \emph{development of $w$ onto $M$ with respect to $\{g(t)\}_{t\in[0,1]}$}.
\end{definition}

Sometimes, the curve $u$ is referred to as the development of $w$, rather than the projection of $u$ onto $M$.


\subsection{Horizontal lift of $g(t)$-Brownian motion}\label{subsection:lift_BM}

In this section we explain how a $g(t)$-Brownian motion may be obtained by solving a stochastic differential equation on $FM$, and projecting the solution down to the manifold. 

 
Malliavin's transfer principle (see e.g. \cite{Mal97}) suggests that constructions for manifold-valued curves can be extended to manifold-valued processes by replacing differential equations by Stratonovich stochastic differential equations. This is because Stratonovich integrals follow the ordinary fundamental theorem of calculus. This suggests that we can obtain a $g(t)$-Brownian motion as the development with respect to $\{g(t)\}_{t\in[0,1]}$ of a standard Brownian motion in $\RR^d$. 

More precisely, we replace the curve $w$ in \eqref{eq:time_development} by a standard $\RR^d$-valued Brownian motion, and interpret the so obtained stochastic differential equation in Stratonovich sense. In symbols this means that for $x_0 \in M$ fixed, we consider the solution $U_t$ of the Stratonovich stochastic differential equation
\begin{equation}\label{eq:SDE_time_BM}
\dd U_t = H_i(t,U_t)\circ \dd W_t^i - \frac12(\partial_1g(t))_{ij}(U_te_i,U_te_j)V^{ij}(U_t)\ud t,
\end{equation}
with $U_0 \in \Oo_{(0,x_0)}$ (see \eqref{eq:set_orthonormal_frames_t}). Here, $H_i(t,u(t)) = H(t,u(t))e_i$ where $H(t,u(t))$ is as defined in \eqref{eq:horizontal_lift_frame_vector} for the Levi-Civita connection $\Nabla^t$ of the metric $g(t)$, and $\{e_1,\ldots,e_d\}$ denotes the standard basis of $\RR^d$. The following is \cite[Proposition 1.4]{Cou11}, see also \cite[Proposition 1.3]{ACT08}.

\begin{proposition}\label{prop:projection_time_BM}
Let $U_t$ be the process on $FM$ solving equation \eqref{eq:SDE_time_BM}. Then $X_t = \pi U_t$ is a $g(t)$-Brownian motion on $M$ starting in $x_0 \in M$. 
\end{proposition}


\section{Proof of Theorem \ref{theorem:Schilder_time} using embeddings} \label{section:proof}

In this section we prove Theorem \ref{theorem:Schilder_time}, the analogue of Schilder's theorem for $g(t)$-Brownian motion. Let us recall the statement of the theorem. 

\begin{theorem}\label{theorem:Schilder_time_proof}
Let $M$ be a Riemannian manifold and let $\{g(t)\}_{t\in[0,1]}$ be a collection of Riemannian metrics, smoothly depending on $t$. Fix $x_0 \in M$, and assume that the $g(t)$-Brownian motion with $X_0 = x_0$ exists for all time $t \in [0,1]$. Furthermore assume that for every $\epsilon > 0$, the continuous process $X_t^\epsilon$ generated by $\frac\epsilon2\Delta_M^t$ exists for all time $t \in [0,1]$. Then $\{X_t^\epsilon\}_{\epsilon > 0}$ satisfies the large deviation principle in $C([0,1];M)$ with good rate function given by
\begin{equation}\label{eq:rate_function_timeBM_proof}
I_M(\gamma) =
\begin{cases}
\frac12\int_0^1 |\dot\gamma(t)|_{g(t)}^2\ud t,		& \gamma \in H_{x_0}^1([0,1];M),\\
\infty									& \mbox{otherwise.}
\end{cases}
\end{equation}
\end{theorem}

As we have seen in Proposition \ref{prop:projection_time_BM}, the horizontal lift $U_t$ with respect to $\{g(t)\}_{t\in[0,1]}$ of a $g(t)$-Brownian motion satisfies the Stratonovich stochastic differential equation
\begin{equation}\label{eq:SDE_time_BM_proof}
\dd U_t = H_i(t,U_t)\circ \dd W_t^i - \frac12(\partial_1g(t))_{ij}(U_te_i,U_te_j)V^{ij}(U_t)\ud t,
\end{equation}
with $U_0 = u_0 \in \Oo_{(0,x_0)}$, where $\Oo_{(0,x_0)}$ is defined in \eqref{eq:set_orthonormal_frames_t}. Similarly, if $\tilde X_t^\epsilon$ is a $g(\epsilon^{-1}t)$-Brownian motion, then its horizontal lift $\tilde U_t^\epsilon$ with respect to $\{g(t)\}_{t\in[0,1]}$ satisfies
$$
\dd \tilde U_t^\epsilon = H_i(\epsilon^{-1}t,\tilde U_t^\epsilon)\circ \dd W_t^i - \frac12(\partial_1g(\epsilon^{-1}t))_{ij}(\tilde U_t^\epsilon e_i,\tilde U_t^\epsilon e_j)V^{ij}(\tilde U_t^\epsilon)\ud t,
$$
with $\tilde U_0^\epsilon = u_0 \in \Oo(0,x_0)$. Finally, the horizontal lift of $X_t^\epsilon = \tilde X_{\epsilon t}^\epsilon$ with respect to $\{g(t)\}_{t\in[0,1]}$ is given by $U_t^\epsilon = \tilde U_{\epsilon t}^\epsilon$. This process satisfies
\begin{equation}\label{eq:horizontal_epsilon_proof}
\dd U_t^\epsilon = H_i(t,U_t^\epsilon)\circ \dd W_t^{\epsilon,i} - \frac12(\partial_tg(t))_{ij}(U_t^\epsilon e_i,U_t^\epsilon e_j)V^{ij}(U_t)\ud t,
\end{equation}
with $U_0 = u_0 \in \Oo_{(0,x_0)}$. Here, $W_t^\epsilon = W_{\epsilon t} = \sqrt\epsilon W_t$. As explained above Theorem \ref{theorem:Schilder_time}, $X_t^\epsilon$ is the rescaled process generated by $\frac{\epsilon}{2}\Delta_M^t$ that we are studying.

The stochastic differential equation for the horizontal lift of $X_t^\epsilon$ obtained in \eqref{eq:horizontal_epsilon_proof} is an important tool for proving Theorem \ref{theorem:Schilder_time}. However, before we can get to this, we first need to make some preparations.

\subsection{Compact containment}\label{subsection:compact_containment}

As part of the proof of Theorem \ref{theorem:Schilder_time}, we need to show that the process $X_t^\epsilon$ generated by $\frac\epsilon2\Delta_M^t$ stays within a compact set with high enough probability when $\epsilon$ tends to 0. In this section we discuss how this can be done via a general approach using Lyapunov functions.

\begin{definition}\label{definition:good_containment_function}
Let $\Hh_t:T^*M \to \RR$ be a collection of maps. A function $\Upsilon:M\to \RR$ is said to be a \emph{good containment function} for the collection $\Hh_t$ if the following are satisfied:
\begin{enumerate}
\item $\Upsilon \geq 0$ and there exists an $x_0 \in M$ such that $\Upsilon(x_0) = 0$.
\item $\Upsilon$ is twice continuously differentiable.
\item For every $c > 0$ the set $\{x \in M| \Upsilon(x) \leq c\}$ is compact.
\item $\sup_{t,x} \Hh_t(x,\dd\Upsilon(x)) < \infty$.
\end{enumerate} 
\end{definition}

We also need to introduce a notion of operator convergence. For this, we first consider bounded and uniform convergence on compact sets (buc), which we define next.

\begin{definition}
Let $\{f_n\}_{n\geq1}$ be a sequence in $C_b(M)$, and let $f \in C_b(M)$. We say that $f_n$ converges to $f$ boundedly, and uniformly on compacts, denoted by $\LIM_{n\to\infty} f_n = f$ if the following are satisfied:
\begin{enumerate}
\item $\sup_n ||f_n|| < \infty$.
\item For all $K \subset M$ compact, 
$$
\lim_{n\to\infty} \sup_{x\in K} |f_n(x) - f(x)| = 0.
$$
\end{enumerate}
\end{definition}

We now define our notion of operator convergence.

\begin{definition}
For every $n \geq 1$, let $A_n: \Dd(A_n) \subset C_b(M) \mapsto C_b(M)$ be an operator. The \emph{extended limit} $ex-\lim_{n\to\infty} A_n$ is defined as the collection $(f,g) \in C_b(M) \times C_b(M)$ for which there exists a sequence $\{f_n\}_{n\geq1}$ with $f_n \in \Dd(A_n)$ such that
$$
\LIM_{n\to\infty} f_n = f, \qquad \LIM_{n\to\infty} A_nf_n = g.
$$
An operator $A$ is said to be contained in $ex-\lim_{n\to\infty} A_n$ if the graph  $\{(f,Af)| f\in \Dd(A)\}$ is a subset of $ex-\lim_{n\to\infty} A_n$. 
\end{definition}

Before we get to the result we are going to use, we first need to define the operators we will be considering.

\begin{assumption}\label{assumption:Hamiltonian}
For every $n \geq 1$, let $A_n^t \subset C_b(M) \times C_b(M)$ be the (time-inhomogeneous) generator of a Markov process $X_n$. Assume that for every $x \in M$, the process $X_n$ started in $x$ is right-continuous and exists for all $t \in [0,1]$. Define the operator 
$$
H_n^tf = \frac1ne^{-nf}A_n^te^{nf}, \qquad e^{nf} \in \Dd(A_n^t).
$$
Suppose that for every $t$, there is an operator $H^t: \Dd(H^t) \subset C_b(M) \to C_b(M)$ with $\Dd(H^t) = C_c^\infty(M)$ and such that $H^t \subset ex-\lim_{n\to\infty} H_n^t$. Finally, assume that $t \mapsto H_t$ is measurable, and that $H^t$ can be written as $H^tf(x) = \Hh^t(x,\dd f(x))$ for some map $\Hh^t:T^*M \to \RR$.
\end{assumption}

The following result is an adaptation of Proposition A.15 in \cite{CK17}, which is based on Lemma 4.22 in \cite{FK06}. The proof is almost verbatim. 

\begin{proposition}\label{prop:containment}
Let Assumption \ref{assumption:Hamiltonian} be satisfied and assume that $X_n(0) = x \in M$ for all $n \geq 1$. Assume that $\Upsilon$ is a good containment function for the collection $\Hh_t$. Assume furthermore that $t \mapsto H_n^t$ is continuous for every $n \geq 1$. Then for every $\alpha > 0$, there exists a compact set $K_\alpha \subset M$ such that
$$
\limsup_{n\to\infty} \frac1n\log\PP\left(X_n(t) \notin K_\alpha \mbox{ for some } t\in [0,1]\right) \leq -\alpha.
$$
Moreover, the sequence $K_\alpha$ can be chosen increasing with $\bigcup_\alpha K_\alpha = M$.
\end{proposition}

\begin{remark}
The continuous dependence of $H_n^t$ on $t$ in Proposition \ref{prop:containment} is used to assure that $\int_0^s H_n^tf(X_n(t))\ud t$ exists. This is necessary to construct a local exponential martingale used in the proof.
\end{remark}


\subsection{Freidlin-Wentzell  theory for time-inhomogeneous diffusions} \label{section:FW_theory}

For the proof of Theorem \ref{theorem:Schilder_time}, we embed the frame bundle $FM$ into some Euclidean space $\RR^N$. Using this embedding, we push forward the stochastic differential equation in \eqref{eq:horizontal_epsilon_proof} to a stochastic differential equation in $\RR^N$. To obtain the large deviations for such diffusions, we use Freidlin-Wentzell theory (\cite{FW12}). Since the stochastic differential equations has time-inhomogeneous coefficients, we have to adjust the Freidlin-Wentzell theory to this case. We follow the line of proof for Freidlin-Wentzell theory for time-homogeneous diffusions, i.e., by using Euler approximations and making the drift and variance constant on small intervals of time, see e.g. \cite[Theorem 5.6.7]{DZ98}. 

\begin{theorem}\label{theorem:time_inhomogeneous_FW}
Let $W_t$ be a standard $\RR^d$-valued Brownian motion and consider for every $\epsilon > 0$ the process $X_t^\epsilon$ satisfying
$$
\dd X_t^\epsilon = b(t,X_t^\epsilon)\ud t + \sqrt{\epsilon}\sigma(t,X_t^\epsilon)\ud W_t,
$$
with $X_0^\epsilon = x_0$.

Assume that $b$ and $\sigma$ are bounded and Lipschitz on $\RR^+\times\RR^d$, i.e.,
$$
|b(t,x) - b(s,y)| + |\sigma(t,x) - \sigma(s,y)| \leq L(|t - s| + |x - y|),
$$
where $L > 0$. Then the sequence $\{X_t^\epsilon\}_{\epsilon > 0}$ satisfies in $C([0,1];\RR^d)$ the large deviation principle with good rate function
$$
I(\gamma) = \inf\left\{\frac12\int_0^1|\dot \varphi(t)|^2\ud t \middle| \gamma(t) = x + \int_0^t b(s,\gamma(s))\ud s + \int_0^t \sigma(s,\gamma(s))\dot\varphi(s)\ud s\right\}.
$$
\end{theorem}

The same result also holds when we consider Stratonovich stochastic differential equations instead of the Ito one. Following the same reasoning as in the proof of \cite[Theorem 2.5]{KRV18}, we have the following corollary.

\begin{corollary}\label{theorem:time_inhomogeneous_FW_Stratonovich}
Let $W_t$ be Brownian motion and consider for every $\epsilon > 0$ the process $X_t^\epsilon$ satisfying the Stratonovich stochastic differential equation
$$
\dd X_t^\epsilon = b(t,X_t^\epsilon)\ud t + \sqrt{\epsilon}\sigma(t,X_t^\epsilon)\circ\ud W_t,
$$
with $X_0^\epsilon = x_0$.

Assume that $b$ and $\sigma$ are bounded and Lipschitz on $\RR^+\times\RR^d$, i.e.,
$$
|b(t,x) - b(s,y)| + |\sigma(t,x) - \sigma(s,y)| \leq L(|t - s| + |x - y|),
$$
where $L > 0$. Then the sequence $\{X_t^\epsilon\}_{\epsilon > 0}$ satisfies in $C([0,1];\RR^d)$ the large deviation principle with good rate function
$$
I(\gamma) = \inf\left\{\frac12\int_0^1|\dot \varphi(t)|^2\ud t \middle| \gamma(t) = x + \int_0^t b(s,\gamma(s))\ud s + \int_0^t \sigma(s,\gamma(s))\dot\varphi(s)\ud s\right\}.
$$
\end{corollary}

\subsection{Proof of Theorem \ref{theorem:Schilder_time}}

Before we prove Theorem \ref{theorem:Schilder_time}, we first need some preliminary results. In the following proposition we prove that given a collection of metrics $\{g(t)\}_{t\in[0,1]}$, we can find another metric that dominates all of these metrics.

\begin{proposition}\label{prop:upper_metric}
Let $\{g(t)\}_{t\in[0,1]}$ be a collection of Riemannian metrics on $M$, depending smoothly on $t$. There exists a Riemannian metric $\overline{g}$ such that for all $x \in M$ and all $v \in T_xM$ we have
$$
g_t(v,v) \leq \overline{g}(v,v)
$$
for all $t \in [0,1]$. 
\end{proposition}
\begin{proof}
Let $\{U_n\}_{n \in \NN}$ be a countable collection of relatively compact charts covering $M$. Furthermore, let $\{\varphi_n\}_{n \in \NN}$ be a partition of unity for the collection $\{U_n\}_{n \in \NN}$.

Writing $G_t(x)$ for the matrix of coordinates of the metric $g(t)$ in a chart $U_n$, we have
\begin{align*}
g(t)(v,v) 
&= 
\left\langle G_t^{\frac12}(x)v,G_t^{\frac12}(x)v\right\rangle_2
\\
&= 
\left\langle G_t^{\frac12}(x)G_0^{-\frac12}(x)G_0^{\frac12}(x)v,G_t^{\frac12}(x)G_0^{-\frac12}(x)G_0^{\frac12}(x)v\right\rangle_2
\end{align*}
for all $v \in T_xM$. Here, the Euclidean inner product has to be understood as the Euclidean inner product of the vector of coefficients of $v$. Using Cauchy-Schwarz inequality, we find 
\begin{equation}\label{eq:inner_product_estimate}
g(t)(v,v) \leq \left|\left|G_t^{\frac12}(x)G_0^{-\frac12}(x)\right|\right|_2^2\left|\left|G_0^{\frac12}(x)v\right|\right|_2^2 = \left|\left|G_t^{\frac12}(x)G_0^{-\frac12}(x)\right|\right|_2^2g(0)(v,v).
\end{equation}
Note that $G_t(x)$ depends continuously on $t$ and $x$, and hence so does $G_t^{\frac12}(x)$. Similarly, $G_0^{-\frac12}(x)$ depends continuously on $x$.  Since $[0,1]$ is compact and $U_n$ is relatively compact, the continuity implies that $||G_t^{\frac12}(x)G_0^{-\frac12}(x)||_2$ is bounded on $[0,1] \times U_n$. If we write 
$$
C = \sup_{t\in [0,1],x\in U_n} \left|\left|G_t^{\frac12}(x)G_0^{-\frac12}(x)\right|\right|_2 < \infty,
$$
then we can define the Riemannian metric $\overline g_n$ on $U_n$ by 
$$
\overline g_n = Cg(0).
$$
From \eqref{eq:inner_product_estimate} it follows that
$$
g_t(v,v) \leq \overline g_n(v,v)
$$
for all $v \in T_xM$ and all $x \in U_n$.

We now define on $M$ the metric 
$$
\overline g = \sum_{n=1}^\infty \varphi_n\overline g_n,
$$
which has the desired property by construction.
\end{proof}

Let us denote by $\overline d$ the Riemannian distance function associated to the metric $\overline g$ from Proposition \ref{prop:upper_metric}. Fix $x_0 \in M$ and consider the radial function $\overline r(x) = \overline d(x,x_0)$. Since $\overline r$ is not everywhere smooth, it is not suitable for constructing a good containment function as in Definition \ref{definition:good_containment_function}. However, since $\overline r$ is 1-Lipschitz (with respect to the metric $\overline g$), we can find a smooth function $\tilde r$ with $\tilde r(x_0) = \overline r(x_0) = 0$ and such that $||\tilde r - \overline r|| \leq 1$ and $|\dd \tilde r|_{\overline g} \leq 2$. Using this, we define $\Upsilon$ by
\begin{equation}\label{eq:containment_function}
\Upsilon(x) = \log(1 + \tilde r(x)^2).
\end{equation}
We now show that $\Upsilon$ can be used as a good containment function for the operators arising from the generator of $g(t)$-Brownian motion. The following is an adaptation of \cite[Lemma 8.2]{KRV18}.

\begin{proposition}\label{prop:good_containment}
Assume $M$ is complete, and let $\{g(t)\}_{t\in[0,1]}$ be a collection of metric on $M$, smoothly depending on $t$. For every $t\in[0,1]$, define $\Hh_tf = \frac12|\dd f|_{g(t)}^2$ for every $f \in C_c^\infty(M)$. Let $\overline g$ be a metric as in Proposition \ref{prop:upper_metric}, and define $\Upsilon$ as in \eqref{eq:containment_function}. Then $\Upsilon$ is a good containment function for the collection $\Hh_t$.
\end{proposition}
\begin{proof}
Clearly, $\Upsilon \geq 0$ and $\Upsilon(x_0) = 0$. Furthermore, since $\tilde r$ is smooth, it follows that $\Upsilon$ is smooth. Now, for $c > 0$, the contintuity of $\Upsilon$ implies that $\{x\in M|\Upsilon(x) \leq c\}$ is closed. Furthermore, $\Upsilon(x) \leq c$ implies that $\tilde r(x) \leq \sqrt{e^c + 1}$. It follows that $\overline d(x,x_0) \leq \tilde r(x) + 1 \leq 1 + \sqrt{e^c + 1}$. Hence, $\{x \in M| \Upsilon(x) \leq c\}$ is bounded. Since $M$ is complete, we conclude that $\{x\in M|\Upsilon(x) \leq c\}$ is compact.

Finally, observe that
$$
\dd\Upsilon(x) = \frac{2\tilde r(x)}{1 + \tilde r(x)^2}\dd\tilde r(x),
$$ 
so that
$$
|\dd\Upsilon(x)|_{\overline g} \leq 2|\dd\tilde r(x)|_{\overline g} \leq 4.
$$
From this, it follows that
$$
\Hh_t(x,\dd\Upsilon(x)) = \frac12|\dd\Upsilon(x)|_{g(t)}^2 \leq \frac12|\dd\Upsilon(x)|_{\overline g}^2 \leq 8.
$$
for all $t$ and $x$. Hence, we find that $\sup_{t,x} \Hh_t(x,\dd\Upsilon(x)) < \infty$.
\end{proof}

We can now show that $X_t^\epsilon$ remains in compact sets with high enough probability.

\begin{proposition}\label{prop:compact_concentration}
Let $M$ be a complete manifold, and let $\{g(t)\}_{t\in[0,1]}$ be a collection of Riemannian metrics on $M$, smoothly depending on $t$. For every $\epsilon > 0$, let $X_t^\epsilon$ be the process generated by $\frac{\epsilon}{2}\Delta_M^t$. Then for every $\alpha \geq 0$, there exists a compact set $K_\alpha \subset M$ such that
$$
\limsup_{\epsilon \to 0} \epsilon\log\PP\left(X_t^\epsilon) \notin K_\alpha \mbox{ for some } t\in[0,1]\right) \leq - \alpha.
$$
Moreover, the sets $K_\alpha$ can be chosen increasing, with $\bigcup_\alpha K_\alpha = M$.
\end{proposition}
\begin{proof}
We apply the results of Section \ref{subsection:compact_containment} with $\epsilon = \frac1n$. Let $f \in C_c^\infty(M)$ and define
$$
H_\epsilon^tf = \epsilon e^{-\epsilon^{-1}f}\frac{\epsilon}{2}\Delta_M^te^{\epsilon^{-1}f}.
$$
Then
$$
H_\epsilon^tf = \epsilon e^{-\epsilon^{-1}f}e^{\epsilon^{-1}f}\frac12(\Delta_M^tf + \epsilon^{-1}|\dd f|^2_{g(t)}) = \frac{\epsilon}{2}\Delta_M^tf + \frac12|\dd f|^2_{g(t)}.
$$
Now define $H^t\subset C_b(M) \times C_b(M)$ with domain $\Dd(H^t) = C_c^\infty(M)$ and $H^tf = \frac12|\dd f|^2_{g(t)}$. Then for all $f \in C_c^\infty(M)$ we have
$$
\lim_{\epsilon\to0} ||H_\epsilon^tf - H^tf|| = 0,
$$
so that $H \subset \LIM_{\epsilon \to 0} H_\epsilon$. Furthermore, note that $H^tf(x) = \Hh^t(x,\dd f(x))$ for $\Hh^t(x,p) = \frac12|p|^2_{g(t)}$. Consequently, Assumption \ref{assumption:Hamiltonian} is fulfilled, and by Proposition \ref{prop:good_containment}, the function $\Upsilon$ given in \eqref{eq:containment_function} is a good containment function for the collection $\Hh_t$. Since $g(t)$ depends continuously on $t$, we find that $t \mapsto H_n^t$ is continuous, so that the claim follows from Proposition \ref{prop:containment}.
\end{proof}

Finally, we also need the following technical lemma.

\begin{lemma}\label{lemma:compact_frames}
Let $M$ be a manifold, and let $\{g(t)\}_{t\in[0,1]}$ be a collection of metrics on $M$, smoothly depending on $t$.  For every $t \in [0,1]$ and $x \in M$, define 
$$
\Oo_{(t,x)} = \{u:\RR^d \to (T_xM,g(t))| u \mbox{ isometry}\}.
$$
Let $K \subset M$ be compact. Then the set 
$$
\bigcup\left\{\Oo_{(t,x)}\middle| t\in[0,1], x \in K\right\}
$$
is a compact subset of $FM$. 
\end{lemma}
\begin{proof}
Consider the bundle $\Oo$ over $\RR \times M$ with fibres $\Oo_{(t,x)}$. For every $(t,x) \in [0,1] \times K$, let $U_{(t,x)} \subset [0,1] \times M$ be open and relatively compact such that there exists a smooth section $u_{(t,x)}$ of $\Oo$ on $\overline U_{(t,x)}$. Since $[0,1] \times K$ is compact, we can find finitely many $(t_1,x_1),\ldots,(t_k,x_k)$ such that 
$$
[0,1] \times K \subset \bigcup_{i=1}^k U_{(t_i,x_i)} \subset \bigcup_{i=1}^k \overline U_{(t_i,x_i)} .
$$
Consequently, we have
$$
\bigcup\left\{\Oo_{(t,x)}\middle| t\in[0,1], x \in K\right\}
\subset
\bigcup_{i=1}^k \bigcup\left\{\Oo_{(t,x)}\middle| t\in[0,1], x \in \overline U_{(t_i,x_i)}\right\}
$$
Since 
$$
\bigcup\left\{\Oo_{(t,x)}\middle| t\in[0,1], x \in K\right\}
$$
is closed, it suffices to show that 
$$
\bigcup\left\{\Oo_{(t,x)}\middle| t\in[0,1], x \in \overline U_{(t_i,x_i)}\right\}
$$
is compact for all $i = 1,\ldots,k$.

For this, consider the map $\Phi_i:[0,1] \times \overline U_{(t_i,x_i)} \times O(d) \to FM$ given by
$$
\Phi_i(t,x,g) = u_{(t_i,x_i)}(t,x)g.
$$
Then $\Phi_i$ is continuous as composition of continuous maps. Furthermore, we have that
$$
\Phi_i([0,1] \times \overline U_{(t_i,x_i)} \times O(d)) = \bigcup\left\{\Oo_{t,x}\middle| t\in[0,1], x \in \overline U_{(t_i,x_i)}\right\}.
$$
Since $[0,1] \times \overline U_{(t_i,x_i)} \times O(d)$ is compact, the above, together with the continuity of $\Phi_i$ now proves the claim.
\end{proof}



With all the preparations done, we are ready to prove Theorem \ref{theorem:Schilder_time}. 

\begin{proof}[Proof of Theorem \ref{theorem:Schilder_time}]
Consider the process $U_t^\epsilon$ in $FM$ defined by 
\begin{equation}\label{eq:epsilon_SDE_proof}
\dd U_t^\epsilon = H_i(t,U_t^\epsilon)\circ \dd W_t^{\epsilon,i} - \frac12(\partial_1g(t))_{ij}(U_t^\epsilon)V^{ij}(U_t^\epsilon)\ud t
\end{equation}
with $U_0^\epsilon = u_0 \in \Oo_{(0,x_0)}$. Here, $W_t^\epsilon = \sqrt{\epsilon}W_t$, where $W_t$ is an $\RR^d$-valued standard Brownian motion. 

Now, let $\{K_\alpha\}_{\alpha > 0}$ be an increasing sequence of compact sets with $\bigcup_\alpha K_\alpha = M$ as in Proposition \ref{prop:compact_concentration}. By Lemma \ref{lemma:compact_frames} we have that
$$
\tilde K_\alpha := \bigcup\left\{\Oo_{(t,x)}\middle|x\in K_\alpha,t\in[0,1]\right\} \subset FM
$$
is compact.

Let $\varphi_\alpha:FM \to \RR$ be a smooth function with compact support and $\varphi \equiv 1$ on $\tilde K_\alpha$. Since $FM$ is locally compact, such a function exists. Consider the process process $U_t^{\epsilon,\alpha}$ in $FM$ given by
$$
\dd U_t^{\epsilon,\alpha} =
\varphi_\alpha(U_t^{\epsilon,\alpha})H_i(t,U_t^{\epsilon,\alpha})\circ \dd W_t^{\epsilon,i} - \frac12\varphi_\alpha(U_t^{\epsilon,\alpha})(\partial_1g(t))_{ij}(U_t^{\epsilon,\alpha})V^{ij}(U_t^{\epsilon,\alpha})\ud t,
$$
with $U_0^{\epsilon,\alpha} = u_0$. 

By Whitney's embedding theorem, there exists an $N \in \NN$ and a smooth embedding $\iota:FM \to \RR^N$. By Ito's formula, we find using \eqref{eq:epsilon_SDE_proof} that
\begin{align*}
&\dd(\iota(U_t^{\epsilon,\alpha}))
\\ 
&= 
\varphi_\alpha(U_t^{\epsilon,\alpha})H_i(t,\cdot)\iota(U_t^{\epsilon,\alpha})\circ \dd W_t^{\epsilon,i} - \frac12\varphi_\alpha(U_t^{\epsilon,\alpha})\left[(\partial_1g(t))_{ij}V^{ij}\right]\iota(U_t^{\epsilon,\alpha})\ud t
\\
&=
\varphi_\alpha(U_t^{\epsilon,\alpha})\iota^*H_i(t,\iota(U_t^{\epsilon,\alpha}))\circ \dd W_t^{\epsilon,i} - \frac12\varphi_\alpha( U_t^{\epsilon,\alpha})\iota^*\left[(\partial_1g(t))_{ij}V^{ij}\right](\iota(U_t^{\epsilon,\alpha}))\ud t.
\end{align*}

Consequently, we obtained a Stratonovich stochastic differential equation for the $\RR^N$-valued process $\tilde U_t^{\epsilon,\alpha} := \iota(U_t^{\epsilon,\alpha})$. Since $\iota$ and $\iota^{-1}$ are smooth, the vector fields
$$
\varphi_\alpha(\iota^{-1}(\tilde U_t^{\epsilon,\alpha}))\iota^*H_i(t,\tilde U_t^{\epsilon,\alpha})
$$
and
$$
\frac12\varphi_\alpha(\iota^{-1}(\tilde U_t^{\epsilon,\alpha}))\iota^*\left[(\partial_1g(t))_{ij}V^{ij}\right](\tilde U_t^{\epsilon,\alpha})
$$
are smooth and compactly supported inside $\iota(FM)$. By putting them equal to zero outside $\iota(FM)$, we obtain smooth, compactly supported vector fields on $\RR^N$. With slight abuse of notation, we denote these vector fields by the same symbol. Consequently, the equation
\begin{multline}
\dd(\tilde U_t^{\epsilon,\alpha}) = \varphi_\alpha(\iota^{-1}(\tilde U_t^{\epsilon,\alpha}))\iota^*H_i(t,\tilde U_t^{\epsilon,\alpha})\circ \dd W_t^{\epsilon,i} \\
- \frac12\varphi_\alpha(\iota^{-1}(\tilde U_t^{\epsilon,\alpha}))\iota^*\left[(\partial_1g(t))_{ij}V^{ij}\right](\tilde U_t^{\epsilon,\alpha})\ud t
\end{multline}
with $\tilde U_0^{\epsilon,\alpha} = \iota(u_0)$ can be considered as equation on $\RR^N$. Since the drift and diffusion are smooth and compactly supported, they are Lipschitz and bounded. By Corollary \ref{theorem:time_inhomogeneous_FW_Stratonovich} we thus obtain that $\{\tilde U_t^{\epsilon,\alpha}\}_{\epsilon \geq0}$ satisfies in $C([0,1];\RR^N)$ the large deviation principle with good rate function $\tilde I_{\RR^N}^\alpha$ given by
\begin{multline}
\tilde I_{\RR^N}^\alpha(\gamma)\\
= \inf\left\{\int_0^1 |\dot \phi(t)|_{\RR^d}^2\ud t \middle | \gamma(0) = \iota(u_0), \dot\gamma(t) =  \varphi_\alpha(\iota^{-1}(\gamma(t)))\iota^*H_i(t,\gamma(t))\dot\varphi^i(t)\right.\\ \left. - \frac12\varphi_\alpha(\iota^{-1}(\gamma(t)))\iota^*\left[(\partial_1g(t))_{ij}V^{ij}\right](\gamma(t)) \right\}
\end{multline}

Now note that $\iota(FM)$ is closed, and by construction it holds that $\{\tilde U_t^{\epsilon,\alpha}\}_{t\in[0,1]} \in C([0,1];\iota(FM))$ almost surely. Furthermore, suppose that $\gamma(0) = \iota(u_0) \in \iota(FM)$, and there exists a curve $\phi$ such that
$$
\dot\gamma(t) = \varphi_\alpha(\iota^{-1}(\gamma(t)))\iota^*H_i(t,\gamma(t))\dot\phi^i(t) - \frac12\varphi_\alpha(\iota^{-1}(\gamma(t)))\iota^*\left[(\partial_1g(t))_{ij}V^{ij}\right](\gamma(t)).
$$
Then, since the vector fields
$$
(\varphi_\alpha\circ\iota^{-1})\iota^*H_i(t,\gamma(t))
$$
and 
$$
\frac12(\varphi_\alpha\circ\iota^{-1})\iota^*\left[(\partial_1g(t))_{ij}V^{ij}\right]
$$
are tangent to $\iota(FM)$ at points of $\iota(FM)$, we find that $\gamma(t) \in \iota(FM)$ for all $t \in [0,1]$ so that $\gamma \in C([0,1];\iota(FM))$. Consequently, if $\gamma \notin C([0,1];\iota(FM))$, then no such $\phi$ exists, and $\tilde I_{\RR^N}^\alpha(\gamma) = \infty$. It now follows from \cite[Lemma 4.1.5]{DZ98} that $\tilde U_t^{\epsilon,\alpha}$ satisfies in $\iota(FM)$ the large deviation principle with good rate function $\tilde I_{\iota(FM)}^\alpha$ given as the restriction of $\tilde I_{\RR^N}^\alpha$ to $C([0,1];\iota(FM))$.

Since $\iota$ is a homeomorphism, the contraction principle (\cite[Theorem 4.2.1]{DZ98}) implies that $U_t^{\epsilon,\alpha} = \iota^{-1}(\tilde U_t^{\epsilon,\alpha})$ satisfies in $C([0,1];FM)$ the large deviation principle with good rate function $I_{FM}^\alpha$ given by 
\begin{align*}
&I_{FM}^\alpha(\eta)
\\
&=
\tilde I_{\iota(FM)}^\alpha(\iota\circ\eta)
\\
&=
\inf\left\{\int_0^1 |\dot \phi(t)|_{\RR^d}^2\ud t \middle | \iota(\eta(0)) = \iota(u_0), \frac{\dd}{\dd t}(\iota\circ\eta)(t) =  \varphi_\alpha(\eta(t))\iota^*H_i(t,\iota(\eta(t)))\dot\phi^i(t)\right.
\\ 
&\qquad \qquad \qquad \left. - \frac12\varphi_\alpha(\eta(t))\iota^*\left[(\partial_1g(t))_{ij}V^{ij}\right](\iota(\eta(t))) \right\}
\\
&=
\inf\left\{\int_0^1 |\dot \phi(t)|_{\RR^d}^2\ud t \middle | \eta(0) = u_0, \dot\eta(t) =  \varphi_\alpha(\eta(t))H_i(t,\eta(t))\dot\phi^i(t)\right.
\\ 
&\qquad \qquad \qquad \left. - \frac12\varphi_\alpha(\eta(t))\left[(\partial_1g(t))_{ij}V^{ij}\right](\eta(t)) \right\}
\end{align*}

Now, since the projection $\pi:FM \to M$ is smooth, again using the contraction principle, we find that $X_t^{\epsilon,\alpha} = \pi(U_t^{\epsilon,\alpha})$ satisfies in $C([0,1];M)$ the large deviation principle with good rate function $I_M^\alpha$ given by
$$
I_M^\alpha(\zeta) = \inf\{I_{FM}^\alpha(\eta)| \pi(\eta) = \zeta\}.
$$

We show how to obtain the desired expression for $I_M^\alpha$. To this end, suppose that $\zeta$ is such that $\zeta(t) \in K_\alpha$ for all $t \in [0,1]$. Suppose that $\eta:[0,1] \to FM$ is such that $\pi\eta = \zeta$ and $I_{FM}^\alpha(\eta) < \infty$. Then $\eta(0) = u_0$ and there exists a $\phi:[0,1]\to\RR^d$ such that
\begin{equation}\label{eq:ode_hor_lift}
\dot\eta(t) = \varphi_\alpha(\eta(t))H_i(t,\eta(t))\dot\phi^i(t) - \frac12\varphi_\alpha(\eta(t))\left[(\partial_1g(t))_{ij}V^{ij}\right](\eta(t)).
\end{equation}
Since $\eta (0) = u_0 \in \Oo_{(0,x_0)}$, the solution $\tilde\eta$ of the equation
$$
\dot{\tilde\eta}(t) = H_i(t,\tilde\eta(t))\dot\phi^i(t) - \frac12\left[(\partial_1g(t))_{ij}V^{ij}\right](\eta(t)),
$$
with $\tilde\eta(0) = u_0$ satisfies $\tilde\eta(t) \in \Oo_{(t,\zeta(t))}$ for all $t \in[0,1]$. Since $\zeta(t) \in K_\alpha$, we find that $\tilde\eta(t) \in \tilde K_\alpha$ for all $t \in [0,1]$. Hence $\varphi_\alpha(\tilde\eta(t)) = 1$ for all $t \in [0,1]$. But then $\tilde \eta(t)$ is also the solution of \eqref{eq:ode_hor_lift}. We conclude that $\eta$ is the unique horizontal lift with respect to $\{g(t)\}_{t\in[0,1]}$ of $\zeta$ with $\eta(0) = u_0$. In that case, $\phi$ is the anti-development with respect to $\{g(t)\}_{t\in[0,1]}$ of $\zeta$ (see Section \ref{subsection:anti_development}), and we have 
$$
|\dot\phi(t)|_{\RR^d} = |\eta(t)\dot\zeta(t)|_{\RR^d} = |\dot\zeta(t)|_{g(t)}.
$$
Consequently, if $\zeta$ is contained in $K_\alpha$, then the rate function reduces to 
$$
I_M^\alpha(\zeta) = \frac12\int_0^1 |\dot\zeta(t)|_{g(t)}^2\ud t.
$$

To conclude the proof, we show that the sequence $\{X_t^\epsilon\}_{\epsilon > 0}$ satisfies in $C([0,1];M)$ the large deviation principle with rate function as given in \eqref{eq:rate_function_timeBM}. To this end, denote by $T^{\epsilon,\alpha}$ the exit time of $X_t^\epsilon$ from $K_\alpha$. Note that
$$
\PP(T^{\epsilon,\alpha} \leq 1) = \PP(X_t^\epsilon \notin K_\alpha \mbox{ for some } t\in[0,1]).
$$
Hence, by the choice of $K_\alpha$, we have
$$
\limsup_{\epsilon \to 0} \epsilon\log\PP(T^{\epsilon,\alpha} \leq 1) \leq -\alpha
$$
Furthermore, the processes $X_t^{\epsilon}$ and $X_t^{\epsilon,\alpha}$ agree up to time $T^{\epsilon,\alpha}$. 

To prove the upper bound, let $F \subset C([0,1];M)$ be closed. Then
\begin{align*}
\PP(X_t^\epsilon \in F)
&=
\PP(X_t^\epsilon \in F|T^{\epsilon,\alpha} > 1)\PP(T^{\epsilon,\alpha} > 1) + \PP(X_t^\epsilon \in F|T^{\epsilon,\alpha} \leq 1)\PP(T^{\epsilon,\alpha} \leq 1)
\\
&\leq
\PP(X_t^\epsilon \in F \wedge T^{\epsilon,\alpha} > 1) + \PP(T^{\epsilon,\alpha} \leq 1)
\\
&=
\PP(X_t^{\epsilon,\alpha} \in F \cap C([0,1];K_\alpha)) + \PP(T^{\epsilon,\alpha} \leq 1).
\end{align*}
Consequently, we find that
\begin{align*}
&\limsup_{\epsilon\to0} \epsilon \log \PP(X_t^\epsilon \in F)
\\
&\leq 
\max\left\{\limsup_{\epsilon\to0}\epsilon\log\PP(X_t^{\epsilon,\alpha} \in F \cap C([0,1];K_\alpha)),\limsup_{\epsilon\to0}\epsilon\log\PP(T^{\epsilon,\alpha} \leq 1)\right\}
\\
&\leq
\max\left\{-\inf_{\gamma \in F\cap C([0,1];K_\alpha)} I_M^\alpha(\gamma),-\alpha\right\}
\\
&\leq
\max\left\{-\inf_{\gamma \in F} I_M,-\alpha\right\}
\end{align*}
Here, the last line follows from the fact that on $C([0,1];K_\alpha)$, $I_M$ and $I_M^\alpha$ coincide. The upper bound now follows by letting $\alpha$ tend to infinity.

For the lower bound, let $G \subset C([0,1];M)$ be open and fix $\gamma \in G$. Let $\delta > 0$ be such that $B(\gamma,\delta) \subset G$. Since the $K_\alpha$ are increasing with $\bigcup_\alpha K_\alpha = M$, we can find $\alpha > 0$ such that $\gamma$ is contained in the interior of $K_\alpha$. Consequently, by possibly shrinking $\delta$, we have that $B(\gamma,\delta) \subset K_\alpha$. But then we find
$$
\PP(X_t^\epsilon \in G) \geq \PP(X_t^\epsilon \in B(\gamma,\delta)) = \PP(X_t^{\epsilon,\alpha} \in B(\gamma,\delta)).
$$
From this it follows that
\begin{align*}
\limsup_{\epsilon \to 0} \epsilon\log\PP(X_t^\epsilon \in G)
&\geq
\limsup_{\epsilon \to 0} \epsilon\log\PP(X_t^{\epsilon,\alpha} \in B(\gamma,\delta))
\\
&\geq
-I_M^\alpha(\gamma)
\\
&=
-I_M(\gamma).
\end{align*}
Here, the second line follows from the first part of the proof, while in the last line we used that $I_M$ and $I_M^\alpha$ coincide on $C([0,1];K_\alpha)$.\\

Finally, to see that $I_M$ is a good rate function, observe that $I_M = \inf_\alpha I_M^\alpha$, and that $I_M^\alpha$ is a good rate function for every $\alpha > 0$. 
\end{proof}

\smallskip

\textbf{Acknowledgement}
The author thanks Anton Thalmaier for his hospitality during a research visit to Luxembourg, and his help in understanding time-dependent geometry. The author also thanks Frank Redig for his thorough reading of the work, and for suggesting various improvements to the readability and clarity of the presented theory and results. The author acknowledges financial support by the Peter Paul Peterich Foundation via the TU Delft University Fund.

\printbibliography

\end{document}